\documentclass[12pt]{amsart}
\usepackage{color,graphics,graphicx,shortvrb}
\usepackage{epsfig}
\usepackage{amssymb,amsmath,amsfonts,bm}
\usepackage{newlfont}
\usepackage{caption}
\usepackage{latexsym}

\usepackage[utf8]{inputenc}

\def\vbar{\mathchoice{\vrule height6.3ptdepth-.5ptwidth.8pt\kern-.8pt}
   {\vrule height6.3ptdepth-.5ptwidth.8pt\kern-.8pt}
   {\vrule height4.1ptdepth-.35ptwidth.6pt\kern-.6pt}
   {\vrule height3.1ptdepth-.25ptwidth.5pt\kern-.5pt}}
\def\fudge{\mathchoice{}{}{\mkern.5mu}{\mkern.8mu}}
\def\bbc#1#2{{\rm \mkern#2mu\vbar\mkern-#2mu#1}}
\def\bbb#1{{\rm I\mkern-3.5mu #1}}
\def\bba#1#2{{\rm #1\mkern-#2mu\fudge #1}}
\def\bb#1{{\count4=`#1 \advance\count4by-64 \ifcase\count4\or\bba
A{11.5}\or \bbb B\or\bbc C{5}\or\bbb D\or\bbb E\or\bbb F \or\bbc
G{5}\or\bbb H\or \bbb I\or\bbc J{3}\or\bbb K\or\bbb L \or\bbb
M\or\bbb N\or\bbc O{5} \or \bbb P\or\bbc Q{5}\or\bbb R\or\bbc
S{4.2}\or\bba T{10.5}\or\bbc U{5}\or \bba V{12}\or\bba
W{16.5}\or\bba X{11}\or\bba Y{11.7}\or\bba Z{7.5}\fi}}
\renewcommand{\Re}{\mathrm{Re\,}}

\def \barr {\begin{array}{l}}
\def \ear {\end{array}}
\def \beq {\begin{equation}}
\def \eeq {\end{equation}}
\def \beqn {\begin{eqnarray}}
\def \eeqn {\end{eqnarray}}
\def \f {\end{document}}

\def\dfrac{\displaystyle\frac}

\newtheorem{lem}{Lemma}
\newtheorem{rem}{Remark}

\newtheorem{theo}{Theorem}

\newcommand{\field}[1]{\mathbb{#1}}
\newcommand{\R}{\field{R}}

\numberwithin{equation}{section}

\begin{document}

\title[Exponential convergence results for a 2D overhead crane with input delays]{Further results on the asymptotic behavior \\
of a 2D overhead crane with input delays:\\
Exponential convergence}

\author{Ka\"{\i}s Ammari}
\address{UR Analysis and Control of PDEs, UR13ES64, Department of Mathematics, Faculty of Sciences of Monastir, University of Monastir, 5019 Monastir, Tunisia}
\email{kais.ammari@fsm.rnu.tn}

\author{Boumedi\`ene Chentouf}
\address{Kuwait University, Faculty of Science, Department of Mathematics, Safat 13060, Kuwait}
\email{chenboum@hotmail.com, boumediene.chentouf@ku.edu.kw}

\begin{abstract}
This article is concerned with the asymptotic behavior of a 2D overhead crane. Taking into account the presence of a delay in the boundary, and assuming that no displacement term appears in the system, a distributed (interior) damping feedback law is proposed in order to compensate the effect of the delay. Then, invoking the frequency domain method, the solutions of the closed-loop system are proved to converge exponentially to a stationary position. This improves the recent result obtained in \cite{ac}, where the rate of convergence is at most of polynomial type.
\end{abstract}


\maketitle

\thispagestyle{empty}


\section{Introduction}

\setcounter{equation}{0}
\begin{figure}[h]
\centering
\includegraphics[scale=0.5]{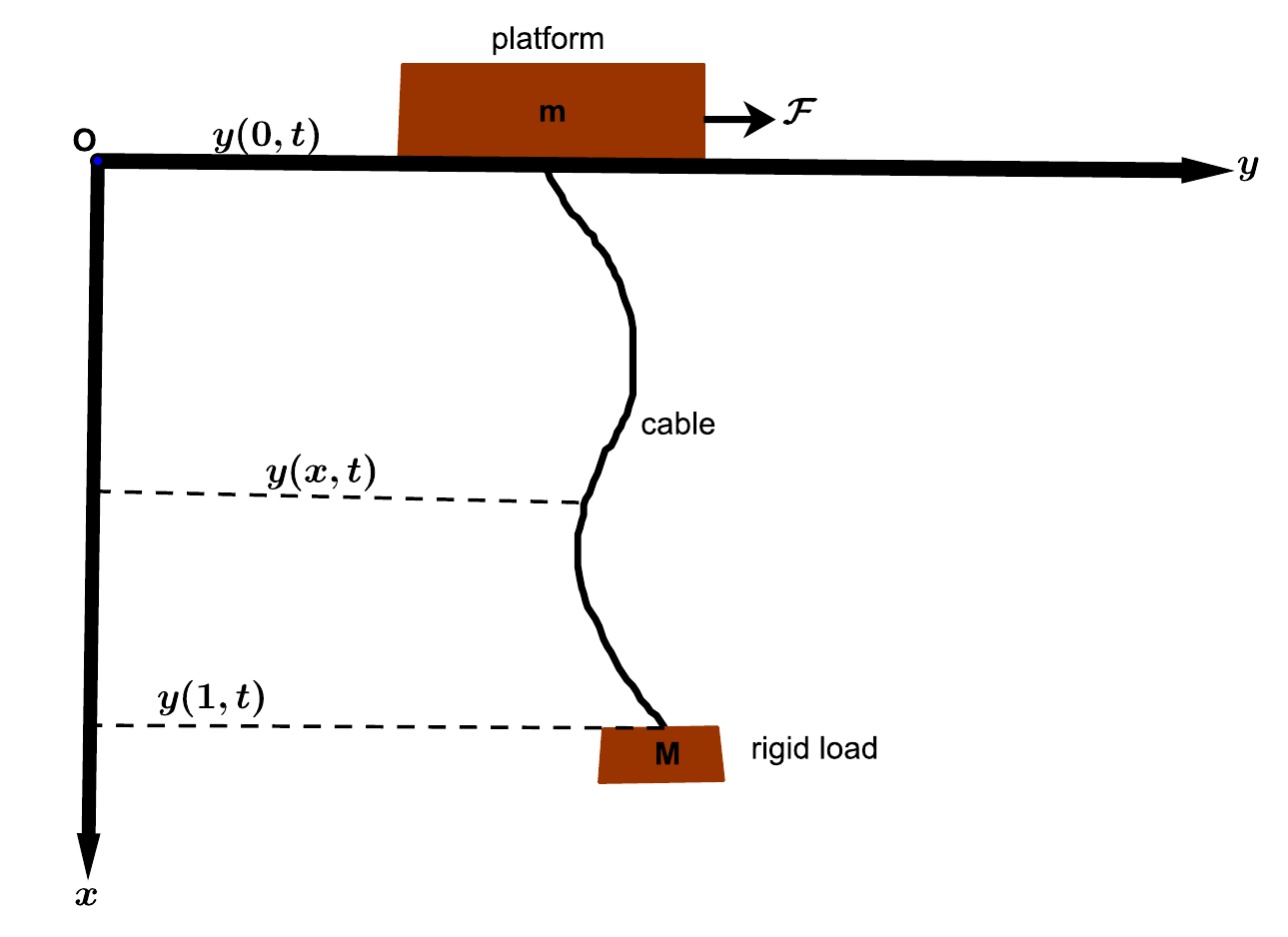}
\caption{The overhead crane model}
\label{fi1}
\end{figure}

In the present work, we consider an overhead crane system (see Fig. \ref{fi1}). It consists of a motorized platform of mass $m$ moving along a horizontal rail, where a flexible cable of unit length is attached. Moreover, the cable is assumed to hold a load mass $M$ to be transported. Taking into consideration the effect of time-delay in the boundary,  the system is governed by the following coupled PDE-ODEs (for more details about the model, the reader is referred to \cite{ac,ANBCR})
\begin{equation}
\left\{
\begin{array}
[c]{ll}%
y_{tt}(x,t)-\left(  ay_{x}\right)  _{x}(x,t)+{\mathcal{F}}(x,t)=0, & 0<x<1,\;t>0,\\
my_{tt}(0,t)-\left(  ay_{x}\right)  (0,t)=\alpha y_{t}(0,t-\tau)-\beta y_{t}(0,t), & t>0,\\
My_{tt}(1,t)+\left(  ay_{x}\right)  (1,t)=0, & t>0,\\
y(x,0)=y_{0}(x),\;y_{t}(x,0)=y_{1}(x), & x \in (0,1),\\
y_{t}(0,t-\tau )=f(t-\tau), & t \in(0,\tau),
\end{array}
\right. \label{(1.1)}%
\end{equation}
in which $y$ is the displacement of the cable,  ${\mathcal{F}}(x,t)$ is the external controlling force (distributed control) which drives the platform along the rail. Furthermore, $\alpha >0, \, \beta>0$, and $\tau >0$ is the time-delay, whereas the space variable coefficient $a(x)$ is the modulus of tension of the cable.

We would like to point out that the feedback gain $\alpha$ is assumed to be positive throughout this article just for sake of simplicity. Otherwise, one can suppose that $\alpha \in \R\setminus \{0\}$ and in this case it suffices to replace $\alpha$ by $|\alpha|$ in the whole paper. In turn, the modulus of tension of the cable $a(x)$ satisfies the following standard conditions (see \cite{ANBCR})
\begin{equation}
\left\{
\begin{array}
[c]{l}%
a\in H^{1}(0,1);\\
\text{there exists a positive constant}\;a_{0}\;\text{such that}\;a(x) \geq
a_{0}>0 \;\;\text{for all}\;\;x \in\lbrack 0,1 \rbrack.
\end{array}
\right.
\label{1.3}
\end{equation}

Overhead cranes have been extensively studied by many authors which gave rise to a huge number of research articles (see for instance \cite{ac,ANBCR,AC,mif2,cos,mif1,Ra}). A pretty comprehensive literature review has been conducted in \cite{ac}, where the reader can have a broad idea about stabilization outcomes achieved in different physical situations. Nonetheless, the effect of the presence of a delay term has been disregarded in those works and this has motivated the authors in \cite{ac} to treat such a case. Specifically, the system (\ref{(1.1)}) with ${\bm {\mathcal F}=0}$ has been considered in \cite{ac}, where it has been showed that the solutions can be driven to an equilibrium state with a polynomial rate of convergence in an appropriate functional space. On the other hand, it has been proved that exponential convergence of solutions cannot be reached when ${\bm {\mathcal F}=0}$.

In this article, we go a step further and improve the convergence result of \cite{ac} through the action of the additional distributed (interior) control
\begin{equation}
{\mathcal F}(x,t)=\sigma y_t(x,t),
\label{F}
\end{equation}
where $\sigma >0$. Specifically, the main contribution of the present work is to show that despite the presence of the delay term in one boundary condition of (\ref{(1.1)}),  the solutions of the closed-loop system (\ref{(1.1)})-(\ref{F}) converge {\bf exponentially} to an equilibrium state which depends on the initial conditions. It is also worth mentioning that another interesting finding of the current work is to prove that the presence of the boundary velocity (in addition of the interior control) is crucial for the convergence of solutions. Indeed, a non-convergence result is established when the boundary velocity is omitted ($\beta =0$ in (\ref{(1.1)})) despite the action of the interior control (\ref{F}). This shows that the ''destabilizing'' effect of the delay term in (\ref{(1.1)}) cannot be compensated through the action of {\bf only} the interior control.

Now, let us briefly present an overview of this paper. In Section \ref{sect2a}, we present a comprehensive study of the closed-loop system (\ref{(1.1)})-(\ref{F}) without boundary velocity ($\beta=0$). The study includes, inter alia, well-posedness and more importantly a non-convergence result for the system with $\beta=0$. In turn, an exponential convergence outcome is proved for a shifted system associated to (\ref{(1.1)})-(\ref{F}) with $\beta=0$. In Section \ref{sect2}, we go back to our initial problem (\ref{(1.1)})-(\ref{F}) and prove that it is indeed well-posed in the sense of semigroups theory of linear operators. LaSalle's principle is used in Section \ref{sect3} to show that the solutions of the closed-loop system (\ref{(1.1)})-(\ref{F}) converge to an equilibrium state whose expression is explicitly determined. In Section \ref{sect4}, the convergence of solutions to the stationary position is showed to be exponential. The proof is based on the frequency domain method. Lastly, the paper ends with a conclusion.

\section{The closed-loop system (\ref{(1.1)})-(\ref{F})  without boundary velocity}
\label{sect2a}
\setcounter{equation}{0}
The main concern of this section is to provide a complete study of the  closed-loop system (\ref{(1.1)})-(\ref{F}) with $\beta=0$, namely,
\begin{equation}
\left\{
\begin{array}
[c]{ll}%
y_{tt}(x,t)-\left(  ay_{x}\right)  _{x}(x,t)+\sigma y_t(x,t)=0, & 0<x<1,\;t>0,\\
my_{tt}(0,t)-\left(  ay_{x}\right)  (0,t)=\alpha y_{t}(0,t-\tau), & t>0,\\
My_{tt}(1,t)+\left(  ay_{x}\right)  (1,t)=0, & t>0,\\
y(x,0)=y_{0}(x),\;y_{t}(x,0)=y_{1}(x), & x \in (0,1),\\
y_{t}(0,t-\tau )=f(t-\tau), & t \in(0,\tau).
\end{array}
\right. \label{(01.1)}%
\end{equation}
Primary emphasis is placed on the study of the asymptotic behavior of solutions to (\ref{(01.1)}). In fact, we shall prove that the solutions of (\ref{(01.1)}) do not converge to their equilibrium state. This outcome justifies why the boundary velocity $y_t(0,t)$ has to be present in (\ref{(01.1)}) if one would like to have a convergence result.

\subsection{Well-posedness of the system (\ref{(01.1)})}
\label{sect02}
\setcounter{equation}{0}

Thanks to the useful change of variables \cite{da} $$u(x,t)=y_{t}(0,t-x\tau),$$ the system (\ref{(01.1)}) writes
\begin{equation}
\left\{
\begin{array}
[c]{ll}%
y_{tt}(x,t)-(ay_{x})_{x}(x,t) + \sigma y_t(x,t) =0, & (x,t)\in(0,1)\times(0,\infty),\\
\tau u_{t}(x,t)+u_{x}(x,t)=0, & (x,t)\in(0,1)\times(0,\infty),\\[1mm]%
my_{tt}(0,t)-\left(  ay_{x}\right)  (0,t)=\alpha u(1,t), & t>0,\\
My_{tt}(1,t)+\left(  ay_{x}\right)  (1,t)=0, & t>0,\\
y(x,0)=y_{0}(x),\;y_{t}(x,0)=y_{1}(x), & x\in(0,1),\\
u(x,0)=y_{t}(0,-x\tau)=f(-x\tau), & x\in(0,1).
\end{array}
\right. \label{03}%
\end{equation}
Then,  consider the state variable $\Phi=(y,z,u,\xi,\eta),$ where $z(\cdot,t)=y_{t}(\cdot,t),\;\xi (t)=y_{t}(0,t)$ and $\eta (t)=y_{t}(1,t)$. Next, the state space of our system is
\[
{\mathcal{H}}_0=H^{1}(0,1)\times L^{2}(0,1)\times L^{2}(0,1)\times\mathbb{R}^{2},
\]
equipped with the following real inner product
\begin{equation}
\begin{array}
[c]{l}%
\langle(y,z,u,\xi,\eta),(\tilde{y},\tilde{z},\tilde{u},\tilde{\xi},\tilde
{\eta})\rangle_{\mathcal{H}_0}=\displaystyle\int_{0}^{1}\left(  ay_{x}\tilde
{y}_{x}+z\tilde{z}\right) \, dx+K \tau\int_{0}^{1} u \tilde{u} \,dx+m \xi \tilde{\xi
}+M\eta\tilde{\eta}\\
+\displaystyle \varpi \left[  \int_{0}^{1}(\sigma y+z) dx+m \xi+M \eta-\alpha y(0)+\tau
\alpha\int_{0}^{1} u \, dx \right] \\
\hspace{5.7cm} \displaystyle \times \left[  \int_{0}^{1} (\sigma \tilde{y}+\tilde{z}) \, dx+m \tilde{\xi
}+M \tilde{\eta}-\alpha{\tilde{y}}(0)+\tau \alpha \int_{0}^{1}\tilde{u} \, dx \right],
\label{0ip}%
\end{array}
\end{equation}
where $K>0$ satisfies the condition
\begin{equation}
\alpha \leq K,\label{0uni}%
\end{equation}
and $\varpi$ is a positive constant sufficiently small so that the norm induced by  (\ref{0ip}) is equivalent to the usual norm of ${\mathcal{H}}_0$. The proof of this claim runs on much the same lines as that of the article \cite{ac} and hence the details are omitted.

Now, one can write the closed-loop system (\ref{03}) as follows
\begin{equation}
\left\{
\begin{array}
[c]{l}%
{\Phi}_t (t)={\mathcal{A}_0}\Phi(t),\\
\Phi(0)=\Phi_{0},
\end{array}
\right. \label{0si}%
\end{equation}
in which $\Phi=(y,z,u,\xi,\eta)$, $\Phi_{0}=(y_{0},y_{1},f(-\tau\cdot),\xi_{0},\eta_{0})$ and
 ${\mathcal{A}_0}$ is a linear operator defined by
\begin{equation}%
\begin{array}
[c]{l}%
{\mathcal{D}}({\mathcal{A}_0})=\left\{  (y,z,u,\xi,\eta)\in{\mathcal{H}}_0;y\in
H^{2}(0,1),\;\;z,u\in H^{1}(0,1),\;\;\xi=u(0)=z(0),\;\eta=z(1)\right\}  ,\\
\displaystyle{{\mathcal{A}_0}}(y,z,u,\xi,\eta)=\left(  z,(ay_{x})_{x}-\sigma z,-\frac{u_{x}}{\tau},\frac{1}{m}\left[  (ay_{x})(0)+\alpha u({1})\right],-\dfrac{(ay_{x})(1)}{M}\right),\label{01.62}
\end{array}
\end{equation}
for each $(y,z,u,\xi,\eta)\in {\mathcal{D}}({\mathcal{A}_0})$.

The well-posedness result is stated below:
\begin{theo}
Assume that (\ref{1.3}) and (\ref{0uni}) hold. Then, the operator ${\mathcal{A}_0}$ defined by (\ref{01.62}) is densely defined in ${\mathcal{H}}_0$ and generates  a $C_{0}$-semigroup $e^{t{\mathcal{A}_0}}$ on ${\mathcal{H}}_0$. Consequently, if $\Phi_{0}\in{\mathcal{D}}({\mathcal{A}_0})$, then the solution $\Phi$ of (\ref{0si}) is strong and belongs to $C([0,\infty);{\mathcal{D}}({\mathcal{A}_0})\cap C^{1}([0,\infty);\mathcal{H}_0)$. However, for any initial condition $\Phi_{0}\in{\mathcal{H}}_0$, the system (\ref{0si}) has a unique mild solution $\Phi\in C([0,\infty);\mathcal{H}_0)$.
\label{(0t1)}
\end{theo}

\begin{proof}
Given $\Phi=(y,z,u,\xi,\eta)\in{\mathcal{D}}({{\mathcal{A}_0}}),$ and using (\ref{0ip}) as well as (\ref{01.62}), we obtain after simple integrations by parts
\[
\langle{{\mathcal{A}_0}}\Phi,\Phi)\rangle_{{\mathcal{H}}_0}=\displaystyle -\sigma \int_{0}^{1}  z^{2} \,dx+(ay_{x}
)(1)z(1)-(ay_{x})(0)z(0)-\dfrac{K}{2}(u^{2}(1)-u^{2}(0))+\xi(ay_{x})(0)+\alpha\xi u(1)
\]%
\[
-\eta(ay_{x})(1)+\varpi\left(  \int_{0}^{1}z\,dx+\tau\alpha\int_{0}%
^{1}u\,dx+m\xi+M\eta\right)  \left(  \alpha u(0)-\alpha z(0)\right)
\]%
\[
=\displaystyle -\sigma \int_{0}^{1}  z^{2} \,dx+\alpha\xi u(1)-\frac{K}{2}u^{2}(1)+\dfrac{K}{2}u^{2}(0).
\]
Applying Young's inequality, we have
\begin{equation}
\langle{{\mathcal{A}_0}}\Phi,\Phi)\rangle_{{\mathcal{H}}_0} \leq \displaystyle -\sigma \int_{0}^{1}  z^{2} \,dx
+\dfrac{K+\alpha}{2}  \xi^{2}+\displaystyle\frac{\alpha-K}{2}u^{2}(1)\label{0dis1}.
\end{equation}
Thenceforth, the operator
\begin{equation}
{\mathcal{P}}:={{\mathcal{A}_0}}-\dfrac{K+\alpha}{2}I
\label{pp}
\end{equation}
is dissipative due to the assumption (\ref{0uni}).

Now, let us show that the operator $\lambda I-{\mathcal{P}}$ is onto $\mathcal{H}_0$ for $\lambda>0$. To do so, let $(f,g,v,p,q)\in{\mathcal{H}}_0$, and let us look for $(y,z,u,\xi,\eta)\in{\mathcal{D}}({\mathcal{A}_0})$ such that $(\lambda I-{{\mathcal{A}_0}})(y,z,u,\xi,\eta)=(f,g,v,p,q)$. This problem reduces to
\begin{equation}
\left\{
\begin{array}
[c]{l}%
\lambda (\lambda+\sigma) y-(ay_{x})_{x}=(\lambda+\sigma) f+g,\\
\lambda\left[ m\lambda-\alpha e^{-{\tau}\lambda}\right]
y(0)-(ay_{x})(0)=mp+(m\lambda-\alpha e^{-{\tau}\lambda})f(0)+\displaystyle{\tau}\alpha\int_{0}^{1}e^{-{\tau}\lambda
(1-s)}v(s)\,ds,\\
\lambda^{2}My(1)+(ay_{x})(1)=Mq+\lambda Mf(1),
\end{array}
\right. \label{0ra1}%
\end{equation}
whose weak formulation is
\[
\displaystyle\int_{0}^{1}\left[  (\lambda+\sigma) \phi y+ay_{x}\phi_{x}\right] \, dx+\lambda\left[  m\lambda-\alpha e^{-{\tau}\lambda}\right]
y(0)\phi(0)+\lambda^{2}My(1)\phi(1)
\]%
\[
=\int_{0}^{1}\left[ (\lambda+\sigma) f+g \right]  \phi\,dx+\left[  mp+(m\lambda -\alpha e^{-{\tau}\lambda})f(0)+\displaystyle{\tau}\alpha\int_{0}%
^{1}e^{-{\tau}\lambda(1-s)}v(s)\,ds\right]  \phi(0)
\]%
\[+\left[  Mq+\lambda Mf(1)\right]  \phi(1).\]
Applying Lax-Milgram Theorem \cite{br}, there exists a unique solution $y\in H^{2}(0,1)$ of (\ref{0ra1}), for $\lambda>0$ and hence the range of $\lambda I-{\mathcal{A}_0}$ is ${\mathcal{H}}_0$. Now, since  ${\mathcal{P}}:={{\mathcal{A}_0}}-\dfrac{K+\alpha}{2}I$, one can deduce that $\lambda I-{\mathcal{P}}$ is also onto ${\mathcal{H}}_0$. Combining the latter with (\ref{0dis1}) and evoking semigroups theory \cite{Pa:83}, the operator ${\mathcal{P}}$ generates generates on ${\mathcal{H}}_0$ a $C_{0}$-semigroup of contraction $e^{t {\mathcal{P}}}$. Lastly, utilizing classical results of bounded perturbations \cite{Pa:83}, we conclude that ${\mathcal{A}_0}$ also generates a $C_{0}$-semigroup on ${\mathcal{H}}_0$. The other assertions of the theorem directly follow from semigroups theory of linear operators \cite{Pa:83}.
\end{proof}


\begin{rem}
Obviously, $\lambda=0$ is an eigenvalue of ${\mathcal{A}_0}$ whose eigenfunction is $(C,0,0,0,0)$, where $C \in\mathbb{R}
\setminus\{0\}$. Thus, the system (\ref{0si}) is not stable. Notwithstanding, one may expect a convergence result of the solutions of (\ref{0si}) to some equilibrium state $(C,0,0,0,0)$ as in \cite{ac} (see also \cite{ac2}). It turned out that this is not the case. The proof of this claim will be given in the next section.
\label{rem1}
\end{rem}

\subsection{Lack of convergence}
\label{sect9}
As stated before, we shall prove that the solutions of the solutions of the system (\ref{(01.1)}) do not converge to their equilibrium state $(C,0,0,0,0)$. First, consider the following subspace
\begin{equation}
\tilde{\mathcal{H}}_0 = \left\{ (y,z,u,\xi,\eta) \in {\mathcal{H}}_0; \, \int_{0}^{1} \left( \sigma y(x)+z(x) +\alpha\tau u(x) \right) \, dx -\alpha \, y(0) + m \xi+ M \eta= 0 \right\},
\label{h0}
\end{equation}
and the operator
\[
\tilde{\mathcal{A}}_0 : \mathcal{D}(\tilde{\mathcal{A}}_0) := \mathcal{D}({\mathcal{A}_0})
\cap \tilde{\mathcal{H}}_0 \subset\tilde{\mathcal{H}}_0 \rightarrow \tilde{\mathcal{H}}_0,
\]
\begin{equation}
\label{til0}
\tilde{\mathcal{A}}_0 (y,z,u,\xi,\eta) = \mathcal{A}_0 (y,z,u,\xi,\eta), \; \mbox{for any} \; (y,z,u,\xi,\eta) \in \mathcal{D}(\tilde{\mathcal{A}}_0).
\end{equation}
Based on the outcomes of Section \ref{sect02}, the operator $\tilde{\mathcal{A}}_0$ generates a $C_{0}$-semigroup  $e^{t \tilde{\mathcal{A}}_0}$. Notwithstanding, $0$ is not an eigenvalue of $\tilde{\mathcal{A}}_0$.

Furthermore, in order to reach our target (the solutions of the system (\ref{(01.1)}) do not converge to their equilibrium state $(C,0,0,0,0)$), it suffices to prove that the semigroup $e^{t \tilde{\mathcal{A}}_0}$ generated by the operator $\tilde{\mathcal{A}}_0$ (see (\ref{til0})) is not stable on $\tilde{\mathcal{H}}_0$. To do so,  let assume that $a(x)=1$, for all $x \in [0,1]$. Then, consider a nonzero complex number $\lambda$ and a nonzero function of $f$ in $H^2(0,1)$. The immediate task is to seek a solution $y(x,t)=e^{\lambda t} f(x)$ of the system associated to the operator $\tilde{\mathcal{A}}_0$  on $\tilde{{\mathcal{H}}}_0$, namely,
\begin{equation}
\left\{
\begin{array}[c]{ll}%
y_{tt}(x,t)-y_{xx}(x,t)=0, & 0<x<1,\;t>0,\\
my_{tt}(0,t)-  y_{x}  (0,t)=\alpha y_{t}(0,t-\tau)-\beta y_{t}(0,t), & t>0,\\
My_{tt}(1,t)+ y_{x}  (1,t)=0, & t>0,
\label{1.6f}%
\end{array}
\right.
\end{equation}
with $ \int_{0}^{1} y_t(x,t)\, dx +\alpha \tau \int_{0}^{1} y_{t}(0,t-x \tau) \, dx + (\beta-\alpha) \, y(0) + m y_{t}(0,t)+ M y_{t}(1,t)= 0.$  A straightforward computation gives the following system
\begin{equation}
\left\{
\begin{array}{l}
f_{xx} -\left(\lambda^2+\sigma \lambda\right) f=0,\\
f_{x}(0)+\lambda \left(\alpha e^{-\lambda \tau}-m \lambda \right) f(0)=0,\\
f_{x}(1)+M \lambda^2 f(1)=0.\label{k2}%
\end{array}
\right.
\end{equation}
Solving the differential equation of (\ref{k2}) yields $y(x)=k_1 e^{-\sqrt{\lambda^2+\sigma \lambda} x}+k_2 e^{\sqrt{\lambda^2+\sigma \lambda}}$, where $k_1$ and $k_2$ are constants to be determined via the boundary conditions in (\ref{k2}). The latter leads us to claim that $f$ is a nontrivial solution of (\ref{k2}) if and only if $\lambda$ is a nonzero solution of the following equation:
\begin{eqnarray}
&&  \left[ \lambda \left( \alpha e^{-\lambda \tau} -\lambda m \right) -\sqrt{\lambda^2+\sigma \lambda} \right] \left[M \lambda^2+ \sqrt{\lambda^2+\sigma \lambda} \right] e^{\sqrt{\lambda^2+\sigma \lambda}} \nonumber\\
&-& \left[ \lambda \left( \alpha e^{-\lambda \tau} -\lambda m \right) + \sqrt{\lambda^2+\sigma \lambda} \right] \left[M \lambda^2 - \sqrt{\lambda^2+\sigma \lambda} \right] e^{-\sqrt{\lambda^2+\sigma \lambda}}=0.\label{k3}
\end{eqnarray}
For sake of simplicity, we shall pick up $\lambda=\sigma >0$. Next, we choose $\tau=\sqrt{2}$ and $M=\sqrt{2}/\sigma.$ Thereby, (\ref{k3}) yields
\[\alpha= \sqrt{2} \left(1+\frac{m}{M}\right) e^{\frac{2}{M}}.\]
Whereupon, (\ref{k3}) holds for $\lambda = \sigma >0$ and the above values of the time-delay $\tau$ as well as the physical parameters $M$ and $\alpha$. Thus, $y(x,t)=e^{\sigma t} f(x)$ is a solution of (\ref{1.6f}), which clearly does not tend to zero in $\tilde{{\mathcal{H}}}_0$. To summarize, we have showed the following result:
\begin{theo}
Let $\tau=\sqrt{2}$ and $M=\sqrt{2}/\sigma$ and $\alpha=\sqrt{2} \left(1+\frac{m}{M}\right) e^{\frac{2}{M}}.$ Although  the assumption (\ref{0uni}) hold, the solutions of (\ref{1.6f}) are unstable in $\tilde{{\mathcal{H}}}_0$. In other words,  the solutions of the closed-loop system (\ref{(01.1)}) do not converge in ${\mathcal{H}}_0$ to their equilibrium state as $t\longrightarrow+\infty$. \label{non}
\end{theo}


\subsection{Asymptotic behavior result of $e^{t\mathcal{P}}$}
\label{sect03}
In this section, we are going to study the asymptotic behavior of the semigroup $e^{t\mathcal{P}}$ generated by the operator $\mathcal{P}$ (see (\ref{pp})) which has been evoked in the proof of Theorem \ref{(0t1)}. First, going back to the proof of Theorem \ref{(0t1)}, one can claim that the operator $\left(\lambda I-{\mathcal{P}} \right)  ^{-1}$ exists and maps ${\mathcal{H}}_0$ into ${\mathcal{D}}({\mathcal{P}})$. Thereafter,  Sobolev embedding \cite{ad} leads to assert that $\left(  \lambda I-{\mathcal{P}} \right)^{-1}$ is compact and hence the spectrum of ${\mathcal{P}}$ consists of isolated eigenvalues of finite algebraic multiplicity only \cite{ka}. Moreover, the trajectories set of solutions $\{\Phi(t)\}_{\scriptscriptstyle t\geq0}$ of the shifted system
\begin{equation}
{\Phi}_t (t)={\mathcal{P}} \Phi(t), \;\; \Phi(0)=\Phi_{0}\label{0b}
\end{equation}
is bounded for the graph norm and thus precompact by means of the compactness of the operator $\left(\lambda I-\mathcal{P}\right)  ^{-1}$, for $\lambda>0$. Now, let  $\Phi(t)=\left(  y(t),y_{t} (t),u(t),\xi(t),\eta(t)\right)  =e^{t\mathcal{P}} \Phi_{0}$ be the solution of (\ref{0b})  stemmed from $\Phi_{0}=\left(  y_{0},y_{1},f,\xi_{0},\eta_{0}\right) \in {\mathcal{D}}({\mathcal{P}})$. By virtue of LaSalle's principle \cite{HA}, it follows that the $\omega$-limit set
$$
\omega \left( \Phi_0 \right)=
\left \lbrace \Psi  \in \mathcal{H}; \; \mbox{there exists a sequence} \; t_n \to \infty \;\;\mbox{as}\;\; n \to \infty \; \mbox{such that} \; \Psi=\lim_{{\scriptstyle n \to \infty}} e^{t_n {\mathcal{P}}} \Phi_0  \right \rbrace $$
is non empty, compact, invariant under the semigroup $e^{t{\mathcal{P}}}$. Moreover, the solution $e^{t{\mathcal{P}}}\Phi _{0}\longrightarrow\omega\left(  \Phi_{0}\right)  \;$ as $t\rightarrow \infty\,$ \cite{HA}. Next, let $\tilde{\Phi}_{0}=\left(  \tilde{y}_{0},\tilde {y}_{1},\tilde{f},\tilde{\xi},\tilde{\eta}\right)  \in\omega\left(  \Phi _{0}\right)  \subset{D}({\mathcal{P}})$ and consider $\tilde{\Phi}(t)=\left(
\tilde{y}(t),\tilde{y}_{t}(t),\tilde{u}(t),\tilde{\xi}(t),\tilde{\eta }(t)\right)  =e^{t{\mathcal{P}}}\tilde{\Phi}_{0}\in{D}({\mathcal{P}})$ as the unique strong solution of (\ref{0si}). Since  $\Vert\tilde{\Phi}(t)\Vert_{\mathcal{H}_0}$ is constant \cite[Theorem 2.1.3 p. 18]{HA}, we have
\[
<{\mathcal{P}} \tilde{\Phi},\tilde{\Phi}>_{\mathcal{H}_0}=0.
\]
This, together with (\ref{0dis1}), gives $\tilde{z}=\tilde{y}_t =0$,  and $\tilde{u}(1)=\tilde{y}_{t}(0,t-\tau)=0$ as long as $K > \alpha$. Consequently, $\tilde{y}$ is constant and hence the $\omega$-limit set $\omega\left(  \Phi_{0}\right)  $ reduces to $(\zeta,0,0,0,0)$. Now, it remains to find $\zeta$. To do so, let $(\zeta,0,0,0,0) \in \omega \left(  \Phi_{0}\right)$, which yields
\begin{equation}
\Phi(t_{n})=(y(t_{n}),y_{t}(t_{n}),u(t_{n}),\xi(t_{n}),\eta(t_{n}))=e^{t_{n} {\mathcal{P}}}\Phi_{0}\longrightarrow(\zeta
,0,0,0,0), \, \text{for some} \, t_{n} \rightarrow\infty, \text{as} \, n \rightarrow\infty.
\label{boum1}%
\end{equation}
In turn, any solution of (\ref{0b}) stemmed from $\Phi_{0}=(y_{0},y_{1},f,\xi_{0},\eta_{0})$ verifies
\[
\frac{\text{d}}{\text{dt}}\left[ \int_{0}^{1} \left( \sigma y(x,t)+ y_{t}(x,t)+\alpha \tau y_{t}(0,t-x\tau) \right) \,dx+my_{t}(0,t)+My_{t}(1,t)-\alpha y(0,t) \right]=0,
\]
for each $t\geq0$, and thus
\begin{align}
& \int_{0}^{1} \left( \sigma y(x,t)+ y_{t}(x,t)+\alpha \tau y_{t}(0,t-x\tau) \right)\,dx + my_{t}(0,t)+My_{t}(1,t)-\alpha y(0,t)\nonumber\\
& =\int_{0}^{1} \left( \sigma y(x,0)+ y_{t}(x,0)+\alpha \tau y_{t}(0,-x\tau) \right) \,dx+my_{t}(0,0)+My_{t}(1,0)-\alpha y(0,0)\nonumber\\
& =\int_{0}^{1} \left( \sigma y_0 (x)+ y_{1}(x)+\alpha \tau f(-x\tau) \right) \,dx+m\xi_{0}+M\eta_{0}-\alpha y_{0}(0).\label{bou1}%
\end{align}
Finally, the desired value of $\zeta$ can be obtained by simply letting $t=t_{n}$ in (\ref{bou1}) with $n\rightarrow\infty$, and then using (\ref{boum1}).

Whereupon, the following result has been showed:
\begin{theo}
Assume that (\ref{1.3}) holds and $K$ satisfies $K > \alpha $. Then, for any initial data $\Phi_{0}=(y_{0},y_{1},f,\xi_{0},\eta_{0}) \in {\mathcal{H}}_0$, the solution $\Phi(t)=\biggl(y,y_{t},y_{t}(0,t-x\tau),y_{t}(0,t),y_{t}(1,t)\biggr)$ of (\ref{0b}) tends in ${\mathcal{H}}_0$ to $(\zeta,0,0,0,0)$ as $t\longrightarrow+\infty$, where
\begin{equation}
\zeta=\displaystyle -\frac{1}{\alpha}\displaystyle\left[  \int_{0}^{1} \left( \sigma y_0 (x)+ y_{1} (x)+\alpha \tau f(-\tau x) \right)\,dx-\alpha y_{0}(0)+m\xi _{0}+M\eta_{0}\right]  .\label{0cste}%
\end{equation}
\label{0t2}
\end{theo}

\subsection{Exponential convergence result for $e^{t\mathcal{P}}$}
\label{sect04}

This subsection addresses the problem of determining the rate of convergence of solutions of the shifted system (\ref{0b}) to their equilibrium state $(\zeta,0,0,0,0)$ as $t\longrightarrow+\infty$ previously obtained (see Theorem \ref{0t2}). To proceed, we recall the following result (see Huang \cite{huang} and Pr\"{u}ss \cite{pruss}):

\begin{theo}
\label{lemraokv} There exist two positive constant $C$ and $\omega$ such that a $C_{0}$-semigroup $e^{t{\mathcal{L}}}$ of contractions on a Hilbert space ${\mathcal{H}}$ satisfies the estimate
\[
||e^{t{\mathcal{L}}}||_{{\mathcal{L}}({\mathcal{H}})} \leq C e^{-\omega t}, \; \; \text{for all} \; t>0,
\]
if and only if
\begin{equation}
\rho({\mathcal{L}})\supset\bigr\{i\gamma; \, \gamma\in\mathbb{R}\bigr\}\equiv
i\mathbb{R},\label{01.8kv}%
\end{equation}
and
\begin{equation}
\limsup_{|\gamma|\rightarrow\infty}\Vert(i\gamma I-{\mathcal{L}})^{-1}\Vert_{{\mathcal{L}}({\mathcal{H}})}<\infty,\label{01.9kv}%
\end{equation}
where $\rho({\mathcal{L}})$ denotes the resolvent set of the operator ${\mathcal{L}}$.
\end{theo}

Let us first define another operator on the subspace $\tilde{\mathcal{H}}_0$ (see (\ref{h0})) as follows
\[
\tilde{\mathcal{P}} : \mathcal{D}(\tilde{\mathcal{P}}) := \mathcal{D}({\mathcal{P}})
\cap \tilde{\mathcal{H}}_0 \subset\tilde{\mathcal{H}}_0 \rightarrow\tilde{\mathcal{H}}_0,
\]
\begin{equation}
\label{01.62bis}
\tilde{\mathcal{P}} (y,z,u,\xi,\eta) = {\mathcal{P}} (y,z,u,\xi,\eta), \, \forall\, (y,z,u,\xi,\eta) \in \mathcal{D}(\tilde{\mathcal{P}}).
\end{equation}

Recalling the results obtained in the previous sections, the operator $\tilde{\mathcal{P}}$ defined by (\ref{01.62bis}) generates on $\tilde{{\mathcal{H}}}_0$ a $C_{0}$-semigroup of contractions $e^{t\tilde{\mathcal{P}}}$ and its spectrum $\sigma(\tilde{\mathcal{P}})$ consists of isolated eigenvalues of finite algebraic multiplicity only.

Then, we have the following result:

\begin{theo}
\label{0lrkv} Assume that (\ref{1.3}) hold and $K$ satisfies $K>\alpha$. Then, there exist $C>0$ and $\omega >0$ such that for all $t>0$, we have
\[
\left\Vert e^{t\tilde{\mathcal{P}}}\right\Vert _{{\mathcal{L}}(\tilde{\mathcal{H}}_0)}\leq C e^{-\omega t}.
\]
\end{theo}

One direct consequence of the above theorem is: the solutions of the system (\ref{0b}) exponentially tend in ${\mathcal{H}}_0$ to $(\zeta,0,0,0,0)$ as $t\longrightarrow+\infty$, where $\zeta$ is given by (\ref{0cste}).

\begin{proof}
 In the first lieu, we are going to check that (\ref{01.8kv}) holds. Consider the eigenvalue problem
\begin{equation}
\tilde{\mathcal{P}}Z=i \gamma Z, \label{0eigenkv}%
\end{equation}
where $\gamma \in \mathbb{R}$ and $Z=(y,z,u,\xi,\eta)=0_{\tilde{\mathcal{H}}_0}$. The latter writes
\begin{align}
z -\dfrac{K+\alpha}{2} y & =i\gamma  y\label{0eigen1}\\
(a\,y_{x})_{x} -\left( \sigma+\dfrac{K+\alpha}{2} \right)  z & =i\gamma z\label{0eigen2}\\
-\frac{u_{x}}{\tau}-\dfrac{K+\alpha}{2} u  & =i\gamma u\label{0eigen3}\\
\frac{1}{m}\left[  (ay_{x})(0)+\alpha u(0)\right] -\dfrac{K+\alpha}{2} \xi  & =i\gamma
\xi\ .\label{0eigen4}\\
-\frac{(ay_{x})(1)}{M} -\dfrac{K+\alpha}{2} \eta & =i\gamma\eta,\label{0eigenbis}%
\end{align}
Taking the inner product of (\ref{0eigenkv}) with $Z$ and recalling (\ref{0dis1}), we get:
\begin{equation}
0=\Re\left(  <\tilde{\mathcal{P}}Z,Z>_{{\tilde{\mathcal{H}}_0}}\right)  \leq \displaystyle -\sigma \int_{0}^{1}  z^{2} \,dx-\displaystyle\frac{K-\alpha}{2}u^{2}(1) (\leq0).\label{01.7kv}%
\end{equation}
Thereby, $z=0$ in $L^2(0,1)$, and $u(1)=0$. Amalgamating $z=0$ and (\ref{0eigen1}) gives $y=0$ and hence  $z(1)=\eta$ by (\ref{0eigenbis}). Exploiting (\ref{0eigen3}) together with $u(1)=0$, we get $u=0$. Finally, $\xi=z(0)=0$ by  (\ref{0eigen4}). Whereupon, (\ref{01.8kv}) is fulfilled by the operator $\tilde{\mathcal{P}}$.

We turn now to the proof of the fact that the resolvent operator of $\tilde{\mathcal{P}}$ obeys the condition \eqref{01.9kv}. Suppose that this claim is not true. Then, thanks to Banach-Steinhaus Theorem \cite{br}, there exist a sequence of real numbers $\gamma_{n} \rightarrow+\infty$ and a sequence of vectors $Z^{n}=(y^{n},z^{n},u^{n},\xi^{n},\eta^{n})\in\mathcal{D}(\tilde{\mathcal{P}})$ with
\begin{equation}
\left\| Z^{n} \right\|_{\tilde{\mathcal{H}}_0}=\left\| y^n\right\|_{H^{1}(0,1)}+\left\| z^n \right\|_{L^{2}(0,1)}+\left\| u^n \right\|_{L^{2}(0,1)}+\left| \xi^{n} \right|_{\mathbb{C}}+\left| \eta^{n} \right|_{\mathbb{C}}=1
\label{0boun}%
\end{equation}
such that
\begin{equation}
\Vert(i\gamma_{n}I-\tilde{\mathcal{P}})Z^{n}\Vert_{\tilde
{\mathcal{H}}_0}\rightarrow0\;\;\;\;\mbox{as}\;\;\;n\rightarrow\infty
,\label{01.12kv}%
\end{equation}
that is, {as} $n\rightarrow\infty$, we have:
\begin{equation}
\left( i\gamma_{n} + \dfrac{K+\alpha}{2} \right) y^{n}-z^{n}   \equiv f_{n}\rightarrow0\;\;\;\mbox{in}\;\;H^{1}(0,1),\label{01.13kv}%
\end{equation}%
\begin{equation}
\left( \sigma+i\gamma_{n} + \dfrac{K+\alpha}{2} \right) z^{n}-(a y_x^{n})_{x}  \equiv g^{n} \rightarrow0 \;\;\;\mbox{in}\;\;L^{2}(0,1),\label{01.13bkv}%
\end{equation}%
\begin{equation}
\left( i\gamma_{n} + \dfrac{K+\alpha}{2} \right) u^{n}+\frac{u_x^{n}}{\tau}  \equiv v^{n}\rightarrow0 \;\;\;\mbox{in}\;\;L^{2}%
(0,1),\label{01.14bkv}%
\end{equation}%
\begin{equation}
\left( i\gamma_{n} + \dfrac{K+\alpha}{2} \right) \xi^{n}-\frac{1}{m}\left[  (a y_x^{n})(0)+\alpha u^{n}(1)\right] \equiv p^{n}\rightarrow0 \;\;\;\mbox{in}\;\; \mathbb{C},\label{01.14kv}%
\end{equation}%
\begin{equation}
\left( i\gamma_{n} + \dfrac{K+\alpha}{2} \right) \eta^{n}+\frac{(ay_x^{n})(1)}{M}  \equiv q^{n} \rightarrow0 \;\;\;\mbox{in}\;\; \mathbb{C}.\label{0lasteq}%
\end{equation}
Exploring the fact that
\[ \left\| (i\gamma_{n}I-\tilde{\mathcal{P}})Z^{n} \right\|_{\tilde{\mathcal{H}}_0%
}\geq \left\vert \Re\langle (i\gamma_{n}I-\tilde{\mathcal{P}})Z^{n},Z^{n} \rangle_{\tilde{\mathcal{H}}_0} \right\vert, \]
together with \eqref{01.7kv}-\eqref{01.12kv}, we have
\begin{equation}
z^n \rightarrow0, \quad \gamma_{n} y^n \rightarrow0, \quad \text{and}\; y^n \rightarrow0, \quad \hbox{in}\;L^{2}(0,1)
\label{0z0}%
\end{equation}
and
\begin{equation}
u^{n}(1)\rightarrow0 \;\;\mbox{in}\;\; \mathbb{C}.\label{0unkv}%
\end{equation}
Combining (\ref{01.14bkv}) and (\ref{0unkv}) yields
\begin{equation}
u^{n}\longrightarrow 0\;\;\mbox{in}\;\;L^{2}(0,1).\label{0znkv}%
\end{equation}
It follows from  (\ref{01.13bkv}) that
\[\dfrac{(a y_x^{n})_{x}}{\gamma_{n}}= \left( \dfrac{\sigma+ \dfrac{K+\alpha}{2}}{\gamma_{n}}+i \right) z^n- \dfrac{g^{n}}{\gamma_{n}}.\]
This together with \eqref{0z0} give
\begin{equation}
\dfrac{(a y_x^{n})_{x}}{\gamma_{n}}  \longrightarrow0 \;\;\;\mbox{in}\;\;L^{2}(0,1).
\label{0yxx}%
\end{equation}%
Applying Gagliardo-Nirenberg interpolation inequality \cite{br}, we have
\[
\left\|a y_x^n \right\|_2^2 \leq  c \, \dfrac{\left\|(a y_x^n)_x \right\|_2}{\left| \gamma_{n} \right|}  \left\|\gamma_{n} y^n  \right\|_2,\]
where $\| \cdot \|_2$ is the usual norm in $L^{2}(0,1)$, and $c$ is a positive constant independent of $n$. Thereby,
\begin{equation}
\left\|a y_x^n \right\|_2   \longrightarrow 0 \;\;\;\mbox{in}\;\;L^{2}(0,1),\label{0der}%
\end{equation}
due to \eqref{0yxx} and \eqref{0z0}.

Furthermore, we have
\begin{equation}
\dfrac{(a y_x^{n})(1)}{\gamma_{n}}  \longrightarrow0 \;\;\;\mbox{in}\;\; \mathbb{C},
\label{0yx1}
\end{equation}
by means of \eqref{0boun} and \eqref{0yxx}. Amalgamating \eqref{0lasteq} and \eqref{0yx1}, we have
\begin{equation}
\eta^{n}=z^n (1) \rightarrow0, \;\;\;\mbox{in}\;\; \mathbb{C}.\label{0eta}%
\end{equation}
Similarly, it follows from \eqref{01.14kv} that
\begin{equation}
\xi^{n}=z^n (0) \rightarrow0, \;\;\;\mbox{in}\;\; \mathbb{C}.\label{0xi}%
\end{equation}
To recapitulate, it follows from \eqref{0z0}-\eqref{0znkv}, \eqref{0eta}, and \eqref{0xi} that $\left\Vert
Z^{n}\right\Vert _{\tilde{\mathcal{H}}_0}\longrightarrow 0$, which  contradicts the fact that $\left\Vert Z^{n}\right\Vert _{\tilde{\mathcal{H}}_0}=1,\;\forall\ n\in\mathbb{N}.$ Thus, the conditions (\ref{01.8kv}) and (\ref{01.9kv}) are fulfilled. This completes the proof of Theorem \ref{0lrkv}.
\end{proof}


Now, we go back to our original system (\ref{(1.1)})-(\ref{F}) and investigate in the next sections the existence and uniqueness of its solutions as well as their behavior.

\section{Well-posedness of the closed-loop system (\ref{(1.1)})-(\ref{F})} \label{sect2}
\setcounter{equation}{0}

Letting $u(x,t)=y_{t}(0,t-x\tau),$ the closed-loop system (\ref{(1.1)})-(\ref{F}) becomes
\begin{equation}
\left\{
\begin{array}
[c]{ll}%
y_{tt}(x,t)-(ay_{x})_{x}(x,t) + \sigma y_t(x,t) =0, & (x,t)\in(0,1)\times(0,\infty),\\
\tau u_{t}(x,t)+u_{x}(x,t)=0, & (x,t)\in(0,1)\times(0,\infty),\\[1mm]%
my_{tt}(0,t)-\left(  ay_{x}\right)  (0,t)=\alpha u(1,t)-\beta u(0,t), & t>0,\\
My_{tt}(1,t)+\left(  ay_{x}\right)  (1,t)=0, & t>0,\\
y(x,0)=y_{0}(x),\;y_{t}(x,0)=y_{1}(x), & x\in(0,1),\\
u(x,0)=y_{t}(0,-x\tau)=f(-x\tau), & x\in(0,1).
\end{array}
\right. \label{3}%
\end{equation}
Then,  consider the state variable $\Phi=(y,z,u,\xi,\eta),$ where $z(\cdot,t)=y_{t}(\cdot,t),\;\xi (t)=y_{t}(0,t)$ and $\eta (t)=y_{t}(1,t)$. Next, the state space of our system is
\[
{\mathcal{H}}=H^{1}(0,1)\times L^{2}(0,1)\times L^{2}(0,1)\times\mathbb{R}^{2}.
\]
The immediate task is to equip  ${\mathcal{H}}$ with an inner product that induces a norm equivalent to the usual one. To do so, let $K$ be a positive constant to be determined and define the energy associated to (\ref{3}) as follows:
\begin{equation}
{\mathcal{E}}(t)=E_{0}(t)+E_{1}(t),\label{ene}%
\end{equation}
where
\begin{equation}
E_{0}(t)=\frac{1}{2}\left[ \int_{0}^{1}\left[ y_{t}^{2}(x,t)+a(x)y_{x}^{2}(x,t)+K \tau y_{t}^{2}(0,t-x\tau) \right] \, dx+my_{t}^{2}(0,t)+My_{t}^{2}(1,t) \right]  ,\label{e0}%
\end{equation}
and
\begin{equation}
E_{1}(t)=\frac{1}{2} \left[  \int_{0}^{1}\left[ \sigma y(x,t)+y_{t}(x,t)+\alpha \tau y_{t}(0,t-x\tau)\right] \, dx+my_{t}(0,t)+My_{t}(1,t)+(\beta-\alpha) y(0,t) \right]^2.\label{e1}%
\end{equation}
Differentiating $E_1$ in a formal way, using (\ref{3}) and integrating by parts, we obtain after a straightforward computation
\begin{equation}
{\mathcal{E}}^{\prime}(t)  \leq -\sigma \int_{0}^{1}  y_{t}^{2}(x,t) \,dx + \left(  \displaystyle\frac{K}{2}%
+\dfrac{\alpha}{2}-\beta\right)  y_{t}^{2}(0,t)+\dfrac{1}{2}\left(\alpha -K\right)  y_{t}^{2}(0,t-\tau).\label{dec}%
\end{equation}
Clearly, the energy ${\mathcal{E}}(t)$ is decreasing if we assume that
\begin{equation}
\alpha<\beta,\label{sma}%
\end{equation}
and then choose $K$ such that
\begin{equation}
\alpha \leq K \leq 2\beta-\alpha.\label{uni}%
\end{equation}

Whereupon, the space ${\mathcal{H}}$ shall be equipped with the following real inner product
\begin{equation}
\begin{array}
[c]{l}%
\langle(y,z,u,\xi,\eta),(\tilde{y},\tilde{z},\tilde{u},\tilde{\xi},\tilde
{\eta})\rangle_{\mathcal{H}}=\displaystyle\int_{0}^{1}\left(  ay_{x}\tilde
{y}_{x}+z\tilde{z}\right) \, dx+K \tau\int_{0}^{1} u \tilde{u} \,dx+m \xi \tilde{\xi
}+M\eta\tilde{\eta}\\
+\displaystyle \varpi \left[  \int_{0}^{1}(\sigma y+z) dx+m \xi+M \eta+(\beta-\alpha) y(0)+\tau
\alpha\int_{0}^{1} u \, dx \right] \\
\hspace{5.7cm} \displaystyle \times \left[  \int_{0}^{1} (\sigma \tilde{y}+\tilde{z}) \, dx+m \tilde{\xi
}+M \tilde{\eta}+(\beta-\alpha){\tilde{y}}(0)+\tau \alpha \int_{0}^{1}\tilde{u} \, dx \right],
\label{ip}%
\end{array}
\end{equation}
where $K>0$ satisfies the condition (\ref{uni}) and $\varpi$ is a positive constant sufficiently small so that the norm induced by  (\ref{ip}) is equivalent to the usual norm of ${\mathcal{H}}$. In fact, arguing as in \cite{ac} one can prove the following lemma
\begin{lem}
The state space ${\mathcal{H}}$ endowed with the inner product (\ref{ip}) is a Hilbert space provided that $\varpi$ is small enough and the hypotheses (\ref{1.3}), (\ref{sma}) and (\ref{uni}) are fulfilled. \label{p1}
\end{lem}

Now, one can write the closed-loop system (\ref{3}) as follows
\begin{equation}
\left\{
\begin{array}
[c]{l}%
{\Phi}_t (t)=\mathcal{A}\Phi(t),\\
\Phi(0)=\Phi_{0},
\end{array}
\right. \label{si}%
\end{equation}
in which  $\mathcal{A}$ is a linear operator defined by
\begin{equation}%
\begin{array}
[c]{l}%
{\mathcal{D}}(\mathcal{A})=\left\{  (y,z,u,\xi,\eta)\in{\mathcal{H}};y\in
H^{2}(0,1),\;\;z,u\in H^{1}(0,1),\;\;\xi=u(0)=z(0),\;\eta=z(1)\right\}  ,\\
\displaystyle{\mathcal{A}}(y,z,u,\xi,\eta)=\left(  z,(ay_{x})_{x}-\sigma z,-\frac{u_{x}}{\tau},\frac{1}{m}\left[  (ay_{x})(0)-\beta\xi+\alpha u({1})\right],-\dfrac{(ay_{x})(1)}{M}\right);\label{1.62}
\end{array}
\end{equation}
whereas $\Phi=(y,z,u,\xi,\eta)$ and $\Phi_{0}=(y_{0},y_{1},f(-\tau\cdot),\xi_{0},\eta_{0})$.

The well-posedness result is stated below:
\begin{theo}
Assume that (\ref{1.3}), (\ref{sma}) and (\ref{uni}) hold. Then, the operator $\mathcal{A}$ defined by (\ref{1.62}) is densely defined in ${\mathcal{H}}$ and generates on ${\mathcal{H}}$ a $C_{0}$-semigroup of contractions $e^{t\mathcal{A}}$. Consequently, if $\Phi_{0}\in{\mathcal{D}}(\mathcal{A})$, then the solution $\Phi$ is strong and belongs to $C([0,\infty);{\mathcal{D}}(\mathcal{A})\cap C^{1}([0,\infty);\mathcal{H})$. However, for any initial condition $\Phi_{0}\in{\mathcal{H}}$, the system (\ref{si}) has a unique mild solution $\Phi\in C([0,\infty);\mathcal{H})$. Finally, the spectrum of $\mathcal{A}$ consists of isolated eigenvalues of finite algebraic multiplicity only. \label{(t1)}
\end{theo}

\begin{proof}
Given $\Phi=(y,z,u,\xi,\eta)\in{\mathcal{D}}({\mathcal{A}}),$ and recalling (\ref{ip}) as well as (\ref{1.62}), we obtain after several integrations by parts
\begin{eqnarray*}
\langle{\mathcal{A}}\Phi,\Phi)\rangle_{{\mathcal{H}}} &=& \displaystyle -\sigma \int_{0}^{1}  z^{2} \,dx+(ay_{x}
)(1)z(1)-(ay_{x})(0)z(0)-\dfrac{K}{2}(u^{2}(1)-u^{2}(0))+\xi(ay_{x} )(0)-\beta\xi^{2}\\
&+& \varpi\left(  \int_{0}^{1}z\,dx+\tau\alpha\int_{0}^{1}u\,dx+m\xi+M\eta\right)  \left(  \alpha u(0)+\beta\xi+(\beta-\alpha)z(0)\right)\\
&+& \alpha\xi u(1)- \eta(ay_{x})(1)\\
&=&\displaystyle -\sigma \int_{0}^{1}  z^{2} \,dx+\alpha\xi u(1)-\frac{K}{2}u^{2}(1)+\dfrac{K}{2}u^{2}(0)-\beta
\xi^{2}.
\end{eqnarray*}
Invoking Young's inequality, we deduce that
\begin{equation}
\langle{\mathcal{A}}\Phi,\Phi)\rangle_{{\mathcal{H}}} \leq \displaystyle -\sigma \int_{0}^{1}  z^{2} \,dx+\left[ -\beta+\dfrac{K+\alpha}{2}\right]  \xi^{2}+\displaystyle\frac{\alpha-K}{2}u^{2}(1)\label{dis1}.
\end{equation}
Thenceforth, the operator ${\mathcal{A}}$ is dissipative due to the assumptions (\ref{sma})-(\ref{uni}).

On the other hand, one can show that the operator $\lambda I-\mathcal{A}$ is onto $\mathcal{H}$ for $\lambda>0$. Indeed, given $(f,g,v,p,q)\in{\mathcal{H}}$, we seek $(y,z,u,\xi,\eta)\in{\mathcal{D}}(\mathcal{A})$ such that $(\lambda I-\mathcal{A})(y,z,u,\xi,\eta)=(f,g,v,p,q)$. A classical argument leads us to solve the following elliptic problem
\begin{equation}
\left\{
\begin{array}
[c]{l}%
\lambda (\lambda+\sigma) y-(ay_{x})_{x}=(\lambda+\sigma) f+g,\\
\lambda\left[  (m\lambda+\beta)-\alpha e^{-{\tau}\lambda}\right]
y(0)-(ay_{x})(0)=mp+(m\lambda+\beta-\alpha e^{-{\tau}\lambda})f(0)\\
\hspace{6.9cm}+\displaystyle{\tau}\alpha\int_{0}^{1}e^{-{\tau}\lambda
(1-s)}v(s)\,ds,\\
\lambda^{2}My(1)+(ay_{x})(1)=Mq+\lambda Mf(1),
\end{array}
\right. \label{ra1}%
\end{equation}
whose weak formulation is
\[
\displaystyle\int_{0}^{1}\left[  (\lambda+\sigma) \phi y+ay_{x}\phi_{x}\right] \, dx+\lambda\left[  (m\lambda+\beta)-\alpha e^{-{\tau}\lambda}\right]
y(0)\phi(0)+\lambda^{2}My(1)\phi(1)
\]%
\[
=\int_{0}^{1}\left[ (\lambda+\sigma) f+g \right]  \phi\,dx+\left[  mp+(m\lambda
+\beta-\alpha e^{-{\tau}\lambda})f(0)+\displaystyle{\tau}\alpha\int_{0}%
^{1}e^{-{\tau}\lambda(1-s)}v(s)\,ds\right]  \phi(0)
\]%
\begin{equation}
+\left[  Mq+\lambda Mf(1)\right]  \phi(1).\label{(1I)}%
\end{equation}
Applying Lax-Milgram Theorem \cite{br}, there exists a unique solution $y\in H^{2}(0,1)$ of (\ref{ra1}) for $\lambda>0$ and hence $\lambda I-\mathcal{A}$ is onto ${\mathcal{H}}$. Therefore,  $\left(\lambda I-{\mathcal{A}}\right)  ^{-1}$ exists and maps ${\mathcal{H}}$ into ${\mathcal{D}}({\mathcal{A}})$. Thereafter,  Sobolev embedding \cite{ad} implies  that $\left(  \lambda I-{\mathcal{A}}\right)^{-1}$ is compact and hence the spectrum of $\mathcal{A}$ consists of isolated eigenvalues of finite algebraic multiplicity only \cite{ka}. In turn, semigroups theory \cite{Pa:83} permits to vindicate the other assertions of our proposition.
\end{proof}


\section{Asymptotic behavior of the closed-loop system (\ref{(1.1)})-(\ref{F})}
\label{sect3}
By virtue of (\ref{1.62}), $\lambda=0$ is an eigenvalue of $\mathcal{A}$ whose eigenfunction is $\nu(1,0,0,0,0)$, where $\nu \in\mathbb{R}\setminus\{0\}$. Hence the energy ${\mathcal{E}}(t)$ (see (\ref{ene})) of the system (\ref{3}) corresponding to $\Phi=\nu(1,0,0,0,0)$ is constant and so will not tend to zero as  $t\longrightarrow+\infty$. In turn, we have the following convergence result
\begin{theo}
Assume that (\ref{1.3}), (\ref{sma}) holds and $K$ satisfies $|\alpha |<K<2\beta-\alpha$. Then, for any initial data $\Phi_{0}=(y_{0},y_{1},f,\xi_{0},\eta_{0})\in{\mathcal{H}}$, the solution $\Phi(t)=\biggl(y,y_{t},y_{t}(0,t-x\tau),y_{t}(0,t),y_{t}(1,t)\biggr)$ of the closed-loop system (\ref{si}) tends in ${\mathcal{H}}$ to $(\zeta,0,0,0,0)$ as $t\longrightarrow+\infty$, where
\begin{equation}
\zeta=\displaystyle\frac{1}{\beta-\alpha}\displaystyle\left[  \int_{0}^{1} \left( \sigma y_0 x)+ y_{1} x)+\alpha \tau f(-\tau x)\right)\,dx+(\beta-\alpha)y_{0}(0)+m\xi
_{0}+M\eta_{0}\right]  .\label{cste}%
\end{equation}
\label{t2}
\end{theo}

\begin{proof}
Let $\Phi(t)=\left(  y(t),y_{t} (t),u(t),\xi(t),\eta(t)\right)  =e^{t\mathcal{A}}\Phi_{0}$ be the solution of
(\ref{3}) stemmed from $\Phi_{0}=\left(  y_{0},y_{1},f,\xi_{0},\eta_{0}\right) \in {\mathcal{D}}({\mathcal{A}})$. It follows from Theorem \ref{(t1)} that the trajectories set of solutions $\{\Phi(t)\}_{\scriptscriptstyle t\geq0}$ is bounded for the graph norm and thus precompact by virtue of the compactness of the operator $\left(\lambda I-\mathcal{A}\right)  ^{-1}$, for $\lambda>0$. Thanks to LaSalle's principle \cite{HA}, it follows that the $\omega$-limit set
$$
\omega \left( \Phi_0 \right)=
\left \lbrace \Psi  \in \mathcal{H}; \; \mbox{there exists a sequence} \; t_n \to \infty \;\;\mbox{as}\;\; n \to \infty \; \mbox{such that} \; \Psi=\lim_{{\scriptstyle n \to \infty}} e^{t_n \mathcal{A}} \Phi_0  \right \rbrace $$
is non empty, compact, invariant under the semigroup $e^{t\mathcal{A}}$. Moreover, the solution $e^{t\mathcal{A}}\Phi _{0}\longrightarrow\omega\left(  \Phi_{0}\right)  \;$ as $t\rightarrow \infty\,$ \cite{HA}. Next, let $\tilde{\Phi}_{0}=\left(  \tilde{y}_{0},\tilde {y}_{1},\tilde{f},\tilde{\xi},\tilde{\eta}\right)  \in\omega\left(  \Phi
_{0}\right)  \subset{D}(\mathcal{A})$ and consider $\tilde{\Phi}(t)=\left(
\tilde{y}(t),\tilde{y}_{t}(t),\tilde{u}(t),\tilde{\xi}(t),\tilde{\eta }(t)\right)  =e^{t\mathcal{A}}\tilde{\Phi}_{0}\in{D}(\mathcal{A})$ as the unique strong solution of (\ref{si}). Exploiting the fact that $\Vert\tilde{\Phi}(t)\Vert_{\mathcal{H}}$ is constant \cite[Theorem 2.1.3 p. 18]{HA}, we have
\begin{equation}
<\mathcal{A}\tilde{\Phi},\tilde{\Phi}>_{\mathcal{H}}=0.\label{e1n}%
\end{equation}
Combining (\ref{dis1}) with (\ref{e1n}), we obtain $\tilde{z}=\tilde{y}_t$, $\tilde{\xi}=\tilde{y}_{t}(0,t)=0$ and $\tilde{u}(1)=\tilde{y}_{t}(0,t-\tau)=0$. Consequently, $\tilde{y}$ is constant and hence the $\omega$-limit set $\omega\left(  \Phi_{0}\right)  $ reduces to $(\zeta,0,0,0,0)$.

Now, the proof will be completed once we find $\zeta$. This can be done by arguing as in the proof of Theorem \ref{0t2}. Indeed, let $(\zeta,0,0,0,0) \in \omega \left(  \Phi_{0}\right)$, which yields
\begin{equation}
\Phi(t_{n})=(y(t_{n}),y_{t}(t_{n}),u(t_{n}),\xi(t_{n}),\eta(t_{n}))=e^{t_{n} \mathcal{A}}\Phi_{0}\longrightarrow(\zeta
,0,0,0,0), \, \text{for some} \, t_{n} \rightarrow\infty, \text{as} \, n \rightarrow\infty.
\label{boum2}%
\end{equation}
In turn, any solution of the closed-loop system (\ref{si}) stemmed from $\Phi_{0}=(y_{0},y_{1},f,\xi_{0},\eta_{0})$ verifies
\begin{equation}
\frac{\text{d}}{\text{dt}}\left[ \int_{0}^{1} \left( \sigma y(x,t)+ y_{t}(x,t)+\alpha \tau y_{t}(0,t-x\tau) \right) \,dx+my_{t}(0,t)+My_{t}(1,t)+(\beta-\alpha)y(0,t) \right]=0,
\label{ga}%
\end{equation}
for each $t\geq0$, and thus
\begin{align}
& \int_{0}^{1} \left( \sigma y(x,t)+ y_{t}(x,t)+\alpha \tau y_{t}(0,t-x\tau) \right)\,dx + my_{t}(0,t)+My_{t}(1,t)+(\beta-\alpha)y(0,t)\nonumber\\
& =\int_{0}^{1} \left( \sigma y(x,0)+ y_{t}(x,0)+\alpha \tau y_{t}(0,-x\tau) \right) \,dx+my_{t}(0,0)+My_{t}(1,0)+(\beta-\alpha)y(0,0)\nonumber\\
& =\int_{0}^{1} \left( \sigma y_0 (x)+ y_{1}(x)+\alpha \tau f(-x\tau) \right) \,dx+m\xi_{0}+M\eta_{0}+(\beta-\alpha)y_{0}(0).\label{bou2}%
\end{align}
Lastly, letting $t=t_{n}$ in (\ref{bou2}) with $n\rightarrow\infty$ and then exploring (\ref{boum2}), one can get the expression of  $\zeta$.
\end{proof}


\section{Exponential convergence rate for the closed-loop system (\ref{(1.1)})-(\ref{F})}
\label{sect4}
The main result of this section is to show that the rate of convergence of the solutions of the system (\ref{3}) is exponential. The proof depends on an essential way on the application of the frequency domain theorem (see Theorem \ref{lemraokv}).

Firstly, let us denote by $\hat{\mathcal{H}}$ the closed subspace of $\mathcal{H}$ and of co-dimension $1$ defined as follows
\[
\hat{\mathcal{H}} = \left\{ (y,z,u,\xi,\eta) \in\mathcal{H}; \, \int_{0}^{1} \left( \sigma y(x)+z(x) +\alpha\tau u(x) \right) \, dx + (\beta-\alpha) \, y(0) + m \xi+ M \eta= 0 \right\}.
\]
Subsequently, we consider a new operator associated to the operator $\mathcal{A}$ (see (\ref{1.62}))
\[
\hat{\mathcal{A}} : \mathcal{D}(\hat{\mathcal{A}}) := \mathcal{D}(\mathcal{A})
\cap\hat{\mathcal{H}} \subset\hat{\mathcal{H}} \rightarrow\hat{\mathcal{H}},
\]
\begin{equation}
\label{1.62bis}\hat{\mathcal{A}} (y,z,u,\xi,\eta) = \mathcal{A} (y,z,u,\xi
,\eta), \, \forall\, (y,z,u,\xi,\eta) \in\mathcal{D}(\hat{\mathcal{A}}).
\end{equation}
By virtue of Theorem \ref{(t1)}, the operator $\hat{\mathcal{A}}$ defined by (\ref{1.62bis}) generates on $\dot{{\mathcal{H}}}$ a $C_{0}$-semigroup of contractions $e^{t\hat{\mathcal{A}}}$ under the conditions (\ref{sma}) and (\ref{uni}). Also, $\sigma(\hat{\mathcal{A}})$, the spectrum of $\hat{\mathcal{A}}$, consists of isolated eigenvalues of finite algebraic multiplicity only.

Our main result is

\begin{theo}
\label{lrkv} Assume that (\ref{1.3}) and (\ref{sma}) hold and $K$ satisfies $\alpha<K<2\beta-\alpha$. Then, there exist $C>0$ and $\omega >0$ such that for all $t>0$
we have
\[
\left\Vert e^{t\hat{\mathcal{A}}}\right\Vert _{{\mathcal{L}}(\hat{\mathcal{H}})}\leq C e^{-\omega t}.
\]
\end{theo}

\vspace{3mm}

\begin{proof} For sake of clarity, we shall prove our result by proceeding by steps.

{\bf Step 1:} The task ahead is to check that $\hat{\mathcal{A}}$ satisfies (\ref{01.8kv}), that is, if $\gamma$ is a real number, then $i\gamma$ is not an eigenvalue of $\hat{\mathcal{A}}$, that is,  the equation
\begin{equation}
\hat{\mathcal{A}}Z=i \gamma Z\label{eigenkv}%
\end{equation}
with $Z=(y,z,u,\xi,\eta)\in\mathcal{D}(\hat{\mathcal{A}})$ and $\gamma \in \mathbb{R}$ has only the trivial solution. In other words, one should verify from \eqref{eigenkv} that the system
\begin{align}
z  & =i\gamma y\label{eigen1}\\
(a\,y_{x})_{x} -\sigma z & =i\gamma z\label{eigen2}\\
-\frac{u_{x}}{\tau}  & =i\gamma u\label{eigen3}\\
\frac{1}{m}\left[  (ay_{x})(0)-\beta\xi+\alpha u(0)\right]   & =i\gamma
\xi\ .\label{eigen4}\\
-\frac{(ay_{x})(1)}{M}  & =i\gamma\eta,\label{eigenbis}%
\end{align}
has only the trivial solution. To do so, if $\gamma=0$, then (\ref{eigen1}) gives $z=0$ and hence $\xi=\eta=0$ as $z(0)=\xi$ and $z(1)=\eta$. In turn, if $\gamma\neq0$ then exploiting (\ref{eigenkv}) and (\ref{dis1}), we obtain
\begin{equation}
0=\Re\left(  <\hat{\mathcal{A}}Z,Z>_{{\hat{\mathcal{H}}}}\right)  \leq \displaystyle -\sigma \int_{0}^{1}  z^{2} \,dx+\left[ -\beta+\dfrac{K+\alpha}{2}\right]  \xi^{2}+\displaystyle\frac{\alpha-K}{2}u^{2}(1)
(\leq0).\label{1.7kv}%
\end{equation}
Hence $z=0$ in $L^2(0,1)$, $z(0)=\xi=0$ and $u(1)=0$. Going back to \eqref{eigenkv}, we deduce our desired result.

{\bf Step 2:} We turn now to prove that the resolvent operator of $\hat{\mathcal{A}}$ obeys the condition \eqref{01.9kv}. Otherwise, Banach-Steinhaus Theorem (see \cite{br}) leads us to claim that there exist a sequence of real numbers $\gamma_{n} \rightarrow+\infty$ and a sequence of vectors \newline $Z^{n}=(y^{n},z^{n},u^{n},\xi^{n},\eta^{n})\in\mathcal{D}(\hat{\mathcal{A}})$ with
\begin{equation}
\left\| Z^{n} \right\|_{\hat{\mathcal{H}}}=\left\| y^n\right\|_{H^{1}(0,1)}+\left\| z^n \right\|_{L^{2}(0,1)}+\left\| u^n \right\|_{L^{2}(0,1)}+\left| \xi^{n} \right|_{\mathbb{C}}+\left| \eta^{n} \right|_{\mathbb{C}}=1
\label{boun}%
\end{equation}
such that
\begin{equation}
\Vert(i\gamma_{n}I-\hat{\mathcal{A}})Z^{n}\Vert_{\dot
{\mathcal{H}}}\rightarrow0\;\;\;\;\mbox{as}\;\;\;n\rightarrow\infty
,\label{1.12kv}%
\end{equation}
that is, {as} $n\rightarrow\infty$, we have:
\begin{equation}
i\gamma_{n}y^{n}-z^{n}   \equiv f_{n}\rightarrow0\;\;\;\mbox{in}\;\;H^{1}(0,1),\label{1.13kv}%
\end{equation}%
\begin{equation}
(i\gamma_{n} +\sigma) z^{n}-(a y_x^{n})_{x}  \equiv g^{n} \rightarrow0 \;\;\;\mbox{in}\;\;L^{2}(0,1),\label{1.13bkv}%
\end{equation}%
\begin{equation}
i\gamma_{n}u^{n}+\frac{u_x^{n}}{\tau}  \equiv v^{n}\rightarrow0 \;\;\;\mbox{in}\;\;L^{2}%
(0,1),\label{1.14bkv}%
\end{equation}%
\begin{equation}
i\gamma_{n}\xi^{n}-\frac{1}{m}\left[  (a y_x^{n})(0)-\beta\xi^{n}+\alpha u^{n}(1)\right] \equiv p^{n}\rightarrow0 \;\;\;\mbox{in}\;\; \mathbb{C},\label{1.14kv}%
\end{equation}%
\begin{equation}
i\gamma_{n}\eta^{n}+\frac{(ay_x^{n})(1)}{M}  \equiv q^{n} \rightarrow0 \;\;\;\mbox{in}\;\; \mathbb{C}.\label{lasteq}%
\end{equation}

Exploring the fact that
\[ \left\| (i\gamma_{n}I-\hat{\mathcal{A}})Z^{n} \right\|_{\hat{\mathcal{H}}%
}\geq \left\vert \Re\langle (i\gamma_{n}I-\hat{\mathcal{A}})Z^{n},Z^{n} \rangle_{\hat{\mathcal{H}}} \right\vert, \]
together with \eqref{1.7kv}-\eqref{1.12kv}, we get
\begin{equation}
z^n \rightarrow0, \quad \gamma_{n} y^n \rightarrow0, \quad \text{and}\; y^n \rightarrow0, \quad \hbox{in}\;L^{2}(0,1),
\label{z0}%
\end{equation}
as well as
\begin{equation}
\xi^{n}=z^n (0)=u^{n}(0)\rightarrow0 \; , \quad u^{n}(1)\rightarrow0 \;\;\;\mbox{in}\;\; \mathbb{C}.\label{unkv}%
\end{equation}
In the light of (\ref{1.14bkv}), we have
\[
u^{n}(x)=u^{n}(0)\,e^{-i\tau\gamma_{n}x}+\tau\,\int_{0}^{x}e^{-i\tau\gamma
_{n}(x-s)}v^{n}(s)\,ds.
\]
Recalling that $v^{n}$ converges to zero in $L^{2}(0,1)$ and invoking (\ref{unkv}), the latter yields
\begin{equation}
u^{n}\longrightarrow 0\;\;\mbox{in}\;\;L^{2}(0,1).\label{znkv}%
\end{equation}
Returning to (\ref{1.13bkv}), we get
\[\dfrac{(a y_x^{n})_{x}}{\gamma_{n}}= \left( \dfrac{\sigma}{\gamma_{n}}+i \right) z^n- \dfrac{g^{n}}{\gamma_{n}},\]
which implies by virtue of \eqref{z0}
\begin{equation}
\dfrac{(a y_x^{n})_{x}}{\gamma_{n}}  \longrightarrow0 \;\;\;\mbox{in}\;\;L^{2}(0,1).
\label{yxx}%
\end{equation}%
Furthermore, we have
$$(ay_x^{n})(1)=\int_{x}^{1} (a y_r^{n})_{r} \,dr + a y_x^{n},
$$
and hence
$$\left|(ay_x^{n})(1)\right|^2 \leq c \left( \left\|(a y_x^{n})_{x} \right\|_2^2 +\left\|a y_x^{n} \right\|_2^2 \right),
$$
where $\| \cdot \|_2$ is the usual norm in $L^{2}(0,1)$, and $c$ is a positive constant (independent of $n$). For convenience, we shall use, in the sequel, the same letter $c$ to represent a positive constant which is independent of $n$.
Whereupon,
\begin{equation}
\dfrac{(a y_x^{n})(1)}{\gamma_{n}}  \longrightarrow0 \;\;\;\mbox{in}\;\; \mathbb{C},
\label{yx1}
\end{equation}
by means of \eqref{boun} and \eqref{yxx}. Amalgamating \eqref{1.14kv} and \eqref{yx1}, we have
\begin{equation}
\eta^{n}=z^n (1) \rightarrow0, \;\;\;\mbox{in}\;\; \mathbb{C}.\label{eta}%
\end{equation}
Evoking Gagliardo-Nirenberg interpolation inequality \cite{br}, we have
\[
\left\|a y_x^n \right\|_2^2 \leq  c \, \dfrac{\left\|(a y_x^n)_x \right\|_2}{\left| \gamma_{n} \right|}  \left\|\gamma_{n} y^n  \right\|_2,
\]
for some positive constant $c$. Thereby,
\begin{equation}
\left\|a y_x^n \right\|_2   \longrightarrow 0 \;\;\;\mbox{in}\;\;L^{2}(0,1),\label{der}%
\end{equation}
due to \eqref{yxx} and \eqref{z0}.

Lastly, the findings in \eqref{z0}-\eqref{znkv}, \eqref{eta}, and \eqref{der}  contradicts the fact that $\left\Vert
Z^{n}\right\Vert _{\hat{\mathcal{H}}}=1,\;\forall\ n\in\mathbb{N}.$ Whereupon, we managed to show that the conditions (\ref{01.8kv}) and (\ref{01.9kv}) are fulfilled. This achieves the proof of Theorem \ref{lrkv}.
\end{proof}

Note that one immediate consequence of Theorem \ref{lrkv} is the exponential convergence of the solutions of the closed-loop system (\ref{3}) in ${\mathcal{H}}$ to $(\zeta,0,0,0,0)$ as $t\longrightarrow+\infty$, where $\Omega$ is given by (\ref{cste}).

\subsection{Lack of convergence when $\alpha \geq \beta$}
This subsection is intended to provide an answer to the following question:  what happens to the solutions of the closed-loop system (\ref{3})  if the condition (\ref{sma}) used to get the convergence results of the system (\ref{3}) is violated, that is, if $\alpha \geq \beta$? Obviously, the answer to such a question could be provided once the semigroup $e^{t \hat{\mathcal{A}}}$ is showed to be unstable in $\hat{\mathcal{H}}$, for some delays $\tau$. For sake of simplicity and without loss of generality, we shall suppose that $a=1$. Then, let us look for a solution $y(x,t)=e^{\gamma t} g(x)$, with $\gamma$ is a positive real number and $g$ is a nonzero function of $H^2(0,1)$, of the system associated to the operator $\hat{\mathcal{A}}$ (see (\ref{1.62bis}) on $\hat{{\mathcal{H}}}$, namely,
\begin{equation}
\left\{
\begin{array}[c]{ll}%
y_{tt}(x,t)-y_{xx}(x,t)+\sigma y_t (x,t) =0, & 0<x<1,\;t>0,\\
my_{tt}(0,t)-  y_{x}  (0,t)=\alpha y_{t}(0,t-\tau)-\beta y_{t}(0,t), & t>0,\\
My_{tt}(1,t)+ y_{x}  (1,t)=0, & t>0,
\label{n1.6f}%
\end{array}
\right.
\end{equation}
with $ \int_{0}^{1} y_t(x,t)\, dx +\alpha \tau \int_{0}^{1} y_{t}(0,t-x \tau) \, dx + (\beta-\alpha) \, y(0) + m y_{t}(0,t)+ M y_{t}(1,t)= 0.$ Then, one can claim $\|y\|_{L^2(0,1)} =e^{\gamma t} \|g\|_{L^2(0,1)} \rightarrow +\infty$ and thereby the solution of (\ref{n1.6f}) is unstable.

One can readily verify that $y(x,t)=e^{\gamma t} g(x)$ is a solution to (\ref{n1.6f}) if $g$ is a nontrivial solution to
\begin{equation}
\left\{
\begin{array}{l}
g_{xx} -\left(\gamma^2+\sigma \gamma\right) g=0,\\
g_{x}(0)+\gamma \left(\alpha e^{-\gamma \tau}-m \gamma -\beta \right) g(0)=0,\\
g_{x}(1)+M \gamma^2 g(1)=0.\label{nk1}%
\end{array}
\right.
\end{equation}
Solving the differential equation of (\ref{nk1}) yields $y(x)=c_1 e^{-\sqrt{\gamma^2+\sigma \gamma} x}+c_2 e^{\sqrt{\gamma^2+\sigma \gamma}}$, where $c_1$ and $c_2$ are constants to be determined by the boundary conditions in (\ref{nk1}). Indeed, the latter implies that $g$ is a nontrivial solution of (\ref{nk1}) if and only if $\gamma$ is a nonzero solution of the following equation:
\begin{eqnarray}
&&\left[ \gamma \left( \alpha e^{-\gamma \tau} -\gamma m -\beta \right) -\sqrt{\gamma^2+\sigma \gamma} \right] \left[M \gamma^2+ \sqrt{\gamma^2+\sigma \gamma} \right] e^{\sqrt{\gamma^2+\sigma \gamma}} \nonumber\\
&-& \left[ \gamma \left( \alpha e^{-\gamma \tau} -\gamma m - \beta\right) + \sqrt{\gamma^2+\sigma \gamma} \right] \left[M \gamma^2 - \sqrt{\gamma^2+\sigma \gamma} \right] e^{-\sqrt{\gamma^2+\sigma \gamma}}=0.\label{nk3}
\end{eqnarray}
Now let us take $\gamma=\sigma >0$ and $M=\sqrt{2}/\sigma.$ Thereafter, (\ref{nk3}) gives
\[
\alpha e^{-\frac{\sqrt{2}}{M} \tau} -\beta -\frac{\sqrt{2}}{M} m -\sqrt{2}=0.
\]
Lastly, solving the above equation, we get $\tau=\frac{M}{\sqrt{2}} \ln \left(\frac{\alpha M}{(\beta+\sqrt{2})M+\sqrt{2}m} \right)>0$  provided that $\alpha \geq \beta + \sqrt{2} \left(1+\frac{m}{M}\right)$. Bearing in mind this choices, we conclude that (\ref{nk3}) is satisfied  with $\gamma = \sigma >0$. Therefore, we have
\begin{theo} If the assumption (\ref{sma}) is not satisfied, then there exists a delay $\tau$ for which the convergence of solutions of the closed-loop system (\ref{3}) does not hold. \label{non2}
\end{theo}

\section{Concluding discussion}

\label{sect6}
This article has addressed the problem of improving the convergence rate of solutions of an overhead crane system modeled by a hyperbolic PDE coupled with two ODEs and subject to the effect of a time-delay in the boundary. First, an interior damping control is proposed. Then, the system is showed to be well-posed in a functional space with an appropriate norm. Next, it is proved that the solutions converge to  an equilibrium state. Last but not least, we show that the convergence rate of solutions is exponential. This finding improves that of the authors in a recent work \cite{ac}, where the convergence is only polynomial.

\section*{Acknowledgement}
This work was supported and funded by Kuwait University, Research Project No. (SM04/17).

\end{document}